\documentclass[11pt]{amsart}
\usepackage{amssymb,amsmath,amsthm,verbatim}
\usepackage{appendix}
\usepackage{fullpage}
\usepackage{xcolor}
\usepackage{url}
\usepackage{mathrsfs}
\usepackage[all]{xy}
  \SelectTips{cm}{10}
  \everyxy={<2.5em,0em>:}
\usepackage{tikz}
\usepackage{picinpar} 

\usepackage{fancyhdr}
\usepackage{ifpdf}
  \ifpdf
    \usepackage{hyperref} 
  \else
    
    \newcommand{\href}[2]{#2}
  \fi

\theoremstyle{plain}
  \newtheorem{lemma}[equation]{Lemma}
  \newtheorem{proposition}[equation]{Proposition}
  \newtheorem{theorem}[equation]{Theorem}
  \newtheorem{corollary}[equation]{Corollary} 
   \newtheorem{conj}[equation]{Conjecture}

\theoremstyle{definition}
  \newtheorem{definition}[equation]{Definition}
  \newtheorem*{notation}{Notation}

\theoremstyle{remark}
  \newtheorem{remark}[equation]{Remark}

\renewcommand{\thesection}{\arabic{section}}
\renewcommand{\theequation}{\thesection.\arabic{equation}}

 \DeclareFontFamily{U}{manual}{}
 \DeclareFontShape{U}{manual}{m}{n}{ <->  manfnt }{}
 \newcommand{\manfntsymbol}[1]{%
    {\fontencoding{U}\fontfamily{manual}\selectfont\symbol{#1}}}

\makeatletter
   \@addtoreset{section}{part}
   \@addtoreset{equation}{section}
   \@addtoreset{footnote}{section}

    {\hspace*{\fill}$\lrcorner$\endgraf\endgroup\end{trivlist}}
 
\makeatother

  \DeclareFontFamily{OT1}{pzc}{}
  \DeclareFontShape{OT1}{pzc}{m}{it}{<-> s * [1.100] pzcmi7t}{}
  \DeclareMathAlphabet{\mathpzc}{OT1}{pzc}{m}{it}

\newif\ifhascomments \hascommentstrue
\ifhascomments
  \newcommand{\anton}[1]{{\color{red}[[\ensuremath{\bigstar\bigstar\bigstar} #1]]}}
  \newcommand{\matt}[1]{{\color{red}[[\ensuremath{\spadesuit\spadesuit\spadesuit} #1]]}}
\else
  \newcommand{\anton}[1]{}
  \newcommand{\matt}[1]{}
\fi

\newcommand{\<}{\langle}
\renewcommand{\>}{\rangle} 

\DeclareMathOperator{\ann}{Ann}
\DeclareMathOperator{\aut}{Aut}
\DeclareMathOperator{\Aut}{\ensuremath{\mathcal{A}\kern-.125em\mathpzc{ut}}}

\newcommand{\C}{\mathcal C}

\renewcommand{\emptyset}{\varnothing}

\newcommand{\E}{\mathcal E}
\DeclareMathOperator{\Endo}{\ensuremath{\mathcal{E}\kern-.125em\mathpzc{nd}}}
\newcommand{\F}{\mathcal F}

\newcommand{\Proj}{\mathcal P}

\DeclareMathOperator{\Hom}{\ensuremath{\mathcal{H}\kern-.125em\mathpzc{om}}}
\newcommand{\id}{\mathrm{id}}

\DeclareMathOperator{\im}{Im}

\renewcommand{\L}{\mathcal L}

\newcommand{\M}{\mathcal M}
\DeclareMathOperator{\MM}{M}
\newcommand{\N}{\mathcal N}

\newcommand{\NN}{\mathbb N}
\renewcommand{\O}{\mathcal O}

\newcommand{\QQ}{\mathbb Q}

\DeclareMathOperator{\rad}{rad}

\renewcommand{\setminus}{\smallsetminus}

\DeclareMathOperator{\spec}{Spec}

\newcommand{\T}{\mathcal T}

\newcommand{\Z}{\mathcal{Z}}
\newcommand{\ZZ}{\mathbb{Z}}

 \def\ari[#1]{\ar@{^(->}[#1]}
 \def\are[#1]{\ar[#1]^{\txt{\'et}}}
 \def\areh[#1]{\ar[#1]|{\txt{$H$-eq}}^{\txt{\'et}}}
 \def\ars[#1]{\ar@{->>}[#1]}
 \newcommand{\dplus}{\ar@{}[d]|{\mbox{$\oplus$}}}
 \newcommand{\dtimes}{\ar@{}[d]|{\mbox{$\times$}}}

\DeclareMathOperator{\supp}{Supp}
\DeclareMathOperator{\Tor}{Tor}
\newcommand{\J}{{\mathcal J}}
\newcommand{\B}{{\mathcal B}}

\newcommand{\Q}{{\mathcal Q}}

\newcommand{\HH}{{\mathcal H}}

\DeclareMathOperator{\shTor}{\mathcal{T}\!\mathit{or}}

\newcommand{\ang}[1]{\langle #1 \rangle}

\title[Dynamical Mordell-Lang Conjecture for Coherent Sheaves]{On a Dynamical Mordell-Lang Conjecture\\for Coherent Sheaves}

\author{Jason P. Bell}
\address{University of Waterloo \\
Department of Pure Mathematics \\
Waterloo, Ontario \\
Canada  N2L 3G1}
\email{jpbell@uwaterloo.ca}

\author{Matthew Satriano}
\address{University of Waterloo \\
Department of Pure Mathematics \\
Waterloo, Ontario \\
Canada  N2L 3G1}
\email{msatriano@uwaterloo.ca}

\author{Susan J. Sierra}
\address{University of Edinburgh\\
School of Mathematics\\
Edinburgh\\
EH9 3FD\\
United Kingdom}
\email{s.sierra@ed.ac.uk}

\thanks{The first- and second-named authors were supported by NSERC grants RGPIN-2016-03632 and RGPIN-2015-05631; the third-named author was  partially supported by 
 EPSRC grant EP/M008460/1.}



\begin{document}

\begin{abstract} We introduce a dynamical Mordell-Lang-type conjecture for coherent sheaves. When the sheaves are structure sheaves of closed subschemes, our conjecture becomes a statement about unlikely intersections. We prove an analogue of this conjecture for affinoid spaces, which we then use to prove our conjecture in the case of surfaces. These results rely on a module-theoretic variant of Strassman's theorem that we prove in the appendix.
\end{abstract}

\maketitle
\tableofcontents

\section{Introduction}
We formulate a generalized dynamical Mordell-Lang conjecture for coherent sheaves. We prove our conjecture for surfaces with an automorphism, as well as for quasi-projective varieties $X$ with an automorphism that lies in an algebraic group acting on $X$. The heart of our argument relies on first proving a variant of our conjecture for affinoid algebras. One of our key tools is a module-theoretic analogue of Strassman's theorem \cite{strassman}; Strassman's original theorem concerns zeros of convergent power series. We believe this result may be of independent interest, so we include it in an appendix.

Before stating our conjecture and results, we begin with a review of the dynamical Mordell-Lang conjecture, which is now a theorem in the case of an \'etale self-map \cite{BGT}. Let $X$ be a quasiprojective variety over an algebraically closed field of characteristic zero, $\Phi : X\to X$ a morphism, and $Y$ a closed subvariety of $X$.  For $n\ge 0$ we let $\Phi^n$ denote the $n$-fold composition $\Phi\circ \cdots \circ \Phi$.  The \emph{dynamical Mordell-Lang conjecture} asserts that for all $x\in X$, the set of natural numbers $n$ for which $\Phi^n(x)\in Y$ is a finite union of infinite arithmetic progressions along with a finite set.  In the case that there are infinitely many natural numbers $n$ for which $\Phi^n(x)\in Y$, the dynamical Mordell-Lang conjecture guarantees the existence of an infinite arithmetic progression of such $n$. As a result, there is some closed subset $Y_0\subseteq Y$ with the property that $\Phi^a(Y_0)\subseteq Y_0$ for some $a \in \ZZ_{\geq 1}$ and such that some iterate of $x$ under $\Phi$ lies in $Y_0$. Indeed, if $P:=m+a\ZZ$ is an infinite arithmetic progression such that $\Phi^n(x)\in Y$ for all $n\in P$, then we can take $Y_0$ to be the closure of the $\Phi^n(x)$ with $n\in P$.

In other words, one can interpret the conjecture as follows: one expects that $Y$ contains only finitely many iterates $\Phi^n(x)$, and this is indeed the case provided there is no compelling geometric reason to the contrary, namely the existence of $Y_0$ as above. When cast in this form, it is natural to try to extend the conjecture beyond the case of points.  To do so, one must first define what it means for two subvarieties to intersect ``as expected.''  Here, one gains some insight from Serre's intersection formula \cite[p.~427]{Hart}, which states that if $X$ is smooth, $Y$ and $Z$ are closed subvarieties of $X$, and $W$ is an irreducible component of the set-theoretic intersection $Y\cap Z$ then the intersection multiplicity of $W$ in the intersection product of $Y$ and $Z$ is given by the alternating sum
\begin{equation} 
\sum_{i=0}^{\infty} (-1)^i {\rm length}_{\mathcal{O}_{X,x}} {\rm Tor}^i_{\mathcal{O}_{X,x}} (\mathcal{O}_{X,x}/I(Y), \mathcal{O}_{X,x}/I(Z)),\end{equation} where 
$x$ is the generic point of $W$ and $I(Y)$ and $I(Z)$ are the ideals given by the set of elements in the local ring that vanish respectively at $Y$ and $Z$.  We remark that this sum is finite since all sufficiently high Tor groups are zero.  We observe that the first term in this summation is the length of the tensor product of $\mathcal{O}_{X,x}/I(Y)\otimes\mathcal{O}_{X,x}/I(Z)=\mathcal{O}_{W,x}$. This is what one expects the intersection multiplicity to be when $Y$ and $Z$ do not intersect in a pathological way.  For example, if $X$ is a smooth complex quasiprojective variety and $Y$ and $Z$ are smooth subvarieties that intersect transversally (i.e., for every $x\in Y\cap Z$ we have $T_xY+T_xZ = T_xX$) then this is exactly what occurs.  In particular, a non-transverse intersection can be detected by the non-vanishing of some higher Tor group appearing in Serre's formula. Thus we see that there is an intimate connection between the vanishing of higher Tor groups and the subvarieties intersecting in an agreeable, or generic, manner.

We say that subschemes $Y$ and $Z$ of an ambient scheme $X$ are \emph{homologically transverse} if the sheaf $\shTor_j^X(\mathcal{O}_Y,\mathcal{O}_Z)=0$ for $j\ge 1$.  In light of the above discussion, we see that one can intuitively think of two subschemes being homologically transverse as saying that their intersection product is what one would naively guess:  the length of the scheme-theoretic intersection. 
  For example, a point $x$ and a proper subvariety $Y$ of an irreducible variety $X$ are homologically transverse if and only if $x\notin Y$; curves in $\mathbb{P}^3$ are homologically transverse if and only if they do not intersect; and two irreducible hypersurfaces in $\mathbb{P}^n$ are homologically transverse if and only if they are not equal.

Taking this perspective, one can now extend the dynamical Mordell-Lang conjecture to the more general setting where one considers two subvarieties $Y$ and $Z$ of an ambient complex quasiprojective variety $X$.  If one has an endomorphism $\Phi$ of $X$ then one would like to understand the set of natural numbers $n$ for which the Zariski closure of $\Phi^n(Y)$ and $Z$ are homologically transverse.  In the case where $Y$ is a point, the original dynamical Mordell-Lang conjecture states that the set of $n$ for which $\Phi^n(Y)$ and $Z$ fail to be homologically transverse is a finite union of infinite arithmetic progressions along with a finite set.  It is natural to expect that this phenomenon extends to the general setting where $Y$ is no longer a point. Moreover, identifying a subvariety $Y$ with its structure sheaf $\O_Y$, we can view the conjecture as a statement about coherent sheaves. Just as in the setting of the original dynamical Mordell-Lang conjecture, we expect that in this more general setting, one only gets infinite arithmetic progressions due to a compelling geometric reason.  We thus make the following conjecture:

\begin{conj}\label{conj: BSS}
Let $X$ be a quasiprojective variety over an algebraically closed field $k$ of characteristic zero, and let $\sigma: X\to X$ be an endomorphism of $X$.  If $\mathcal{M}$ and $\mathcal{N}$ are coherent sheaves on $X$ then for each $i\ge 1$, the set of natural numbers $n$ for which
$$\shTor^X_i((\sigma^n)^*\mathcal{M},\mathcal{N})\neq 0$$ is a finite union of infinite arithmetic progressions up to addition and removal of a finite set.
\end{conj}

This generalizes an earlier conjecture of the third-named author \cite[Conjecture~5.15]{geometric-idealizer}, which dealt with the case where $\shTor^X_i((\sigma^n)^*\mathcal{M},\mathcal{N})$ vanishes generically.

We remark that Conjecture~\ref{conj: BSS} really lies in the intersection of work around the dynamical Mordell-Lang conjecture and the field of unlikely intersections.  To illustrate this, we give the following special case, which was inspired by discussions with Dragos Ghioca and Joe Silverman.  Let $X=(\mathbb{C}^*)^3$ and let $\Phi: X\to X$ be the map $(x,y,z)\mapsto (x^2,y^3,z^5)$.  If $Y$ is the rational curve $\{(t,t,t)\colon t\in \mathbb{C}^*\}$ and $Z$ is a curve given by the zero set of two polynomials $A(x,y,z)$ and $B(x,y,z)$, then $\Phi^n(Y)$ and $Z$ are not homologically transverse if and only if there is some $c\in \mathbb{C}^*$ such that $A(c^{2^n}, c^{3^n},c^{5^n}) = B(c^{2^n}, c^{3^n},c^{5^n})=0$.  Thus, although Conjecture \ref{conj: BSS} is inspired in part by the dynamical Mordell-Lang conjecture, one often must ultimately deal with problems related to ``unlikely intersections'' to obtain the desired conclusion.\footnote{For the experts, we remark that an unlikely intersection in the sense of being non-homologically transverse is similar in spirit, but not equivalent to, an unlikely intersection in the sense of Zannier \cite{zannier}. In the latter sense, we say subvarieties $Y$ and $Z$ of $X$ intersect {\em properly} if every component of $Y \cap Z$ has the expected dimension of $\max(\dim Y + \dim Z - \dim X, 0)$ and otherwise say that $Y$ and $Z$ have {\em unlikely intersection}. As shown in the proof of \cite[Lemma 42.16.1]{Stacks}, if $X$ is non-singular, and $Y$ and $Z$ are Cohen-Macaulay and intersect properly, then $Y$ and $Z$ are homologically transverse. On the hand, if $Y$ is a local complete intersection and $Y$ and $Z$ are homologically transverse, then $Y$ and $Z$ have proper intersection. This follows by looking at the Koszul resolution $K_*$ of $Y$, tensoring with $\O_Z$, and noting that $\Tor_{\geq 1}(\mathcal{O}_Y, \mathcal{O}_Z) = 0$ if and only if the higher homology of $K_* \otimes_{\mathcal{O}_X} \mathcal{O}_Z$ vanishes. By \cite[Lemma 15.27.7]{Stacks}, this implies that the equations locally cutting out $Y$ give a regular sequence on $Z$, which implies that $Y$ and $Z$ have proper intersection.

In particular if $Y$ is a local complete intersection, $Z$ is Cohen-Macaulay, and $X$ is smooth, then $Y$ and $Z$ have proper intersection if and only if they are homologically transverse. For subvarieties which are not Cohen-Macaulay, we can have proper intersection without being homologically transverse, see \cite[Example 42.14.4]{Stacks}.}

We note that the above example only deals with structure sheaves of subvarieties of the ambient variety, and geometrically this is arguably the most interesting case of our conjecture. The reason we have formulated our conjecture more generally in terms of coherent sheaves rather than just structure sheaves of subvarieties is that when computing Tor groups of structure sheaves it is often very useful to work with exact sequences involving coherent sheaves that are not structure sheaves. Said another way, we imagine that any (inductive) proof of our conjecture for the case of structure sheaves will naturally lead one to consider the coherent sheaf formulation.

We hope the above example, which is obviously a very special case of the conjecture, gives some underpinning to our belief that the conjecture is very difficult in general.  Indeed, the dynamical Mordell-Lang conjecture is already a hard question, but in this higher-dimensional variant one must also now wrestle with difficult questions involving unlikely intersections.  In this paper, we restrict our attention to the case where $\sigma$ is an automorphism of $X$.  In this setting we can prove Conjecture \ref{conj: BSS} in two cases: when ${\rm dim}(X)\le 2$ or $\sigma$ acts vis an algebraic group.

\begin{theorem}
\label{thm: main1}
Let $k$ be an algebraically closed field of characteristic zero, let $X$ be a smooth quasiprojective variety over $k$, and let $\sigma: X\to X$ be in ${\rm Aut}_k(X)$.  Assume that at least one of the following holds:
\begin{enumerate}
\item $X$ is a surface; or 
\item $\sigma$ lies in an algebraic group acting as $k$-rational automorphisms of $X$.
\end{enumerate}
If $\mathcal{M}$ and $\mathcal{N}$ are coherent sheaves on $X$ then for each $i\ge 1$ the set of $n\in\ZZ$ for which
$$\shTor^X_i((\sigma^n)^*\mathcal{M},\mathcal{N})\neq 0$$ is a finite union of doubly infinite arithmetic progressions up to addition and removal of finite sets.
\end{theorem}

The second case of Theorem~\ref{thm: main1} follows from the third-named author's work on a general Kleiman-Bertini theorem \cite{S-KB}; this case is proved in \S \ref{sec:Sue}.  The surface case is proved in \S \ref{sec:surface-case}, and relies on  a version of Conjecture \ref{conj: BSS} for affinoid spaces, 
along with the original proof of the cyclic case of the dynamical Mordell-Lang conjecture for \'etale endomorphisms.  The affinoid result is:

\begin{theorem}
\label{thm: main2}
Let $S=K\langle x_1,\ldots ,x_d\rangle$ be the Tate algebra over $K$. Suppose that $\sigma: S\to S$ is a $K$-algebra automorphism satisfying $|\sigma(\overline{x}_i)- \overline{x}_i| < p^{-c}$ for some $c>1/(p-1)$ and for $i=1,\ldots ,d$, where $\overline{x}_i$ denotes the image of $x_i$ in $S$. If $M$ and $N$ are finitely-generated $S$-modules and $i\geq1$, then the set of $n\in\ZZ$ for which $${\rm Tor}_i^S((\sigma^n)^*M, N)\neq 0$$ is, up to addition and removal of a finite set, a finite union of arithmetic progressions of difference $p^r$ for some $r\geq0$.
\end{theorem}

 In fact, we prove Theorem \ref{thm: main2} by establishing a more general result, which we state momentarily. 
 In \S \ref{sec:p-adic-family}, using an analytic arc theorem we show that there is an $S\langle z\rangle$-module $\mathcal{M}$ with the property that $\mathcal{M}\otimes S\<z\>/(z-n) =(\sigma^n)^*M$ for all $n\in\ZZ$. 
In other words, we construct a $p$-adic family of modules $\M$ that interpolates between the iterates $(\sigma^n)^*M$, see Definition \ref{def:p-adic-family-ideals}. Making use of some technical results in \S \ref{subsec:specialization-arg}, we then prove the following result in \S \ref{sec:proof-main-thm2}, which clearly implies Theorem \ref{thm: main2}.

\begin{theorem}
\label{thm: main3}
Adopt the assumptions of Theorem \ref{thm: main2}. Then there exist finitely generated $S\<z\>$-modules $\M$ and $\N$ with the following two properties. 
For $i\geq1$, if $\Tor_i^{S\ang{z}}(\M,\N)=0$, then $\Tor_i^S((\sigma^n)^*M,N)=0$ for all but finitely many $n\in\ZZ$; if $\Tor_i^{S\ang{z}}(\M,\N)\neq0$, then the set of $n\in\ZZ$ for which $\Tor_i^S((\sigma^n)^*M,N)\neq0$ is, up to addition and removal of a finite set, a finite union of arithmetic progressions of difference $p^r$ for some $r\geq0$.
\end{theorem}

Finally, as mentioned earlier, all of these results rely on our variant of Strassman's theorem for modules, which we prove in the appendix:

\begin{theorem}
Let $K$ be a field with a non-Archimedean absolute value $|\, \cdot \,|$ such that $|p|=1/p$, and let $R$ be a subring of the valuation ring of $K$. If $\mathcal{M}$ is a finitely generated $K\langle x_1,\ldots ,x_d,z\rangle$-module, then the set of $c\in R$ for which $\mathcal{M}|_{z=c}=(0)$ is open in $R$. If $R$ is compact, then there exists $\varepsilon>0$ such that up to the addition and removal of finite sets, the set of $c\in R$ for which $\mathcal{M}|_{z=c}=(0)$ is a union of balls in $R$ of radius $\varepsilon$.

In particular, if $R=\mathbb{Z}$, then up to addition and removal of finite sets, the set of $c\in R$ for which $\mathcal{M}|_{z=c}=(0)$ is a finite union of arithmetic progressions with difference $p^r$ for some $r\ge 0$.
\end{theorem}


\noindent\textbf{Notation.} Throughout the paper, we frequently denote $(\sigma^n)^*M$ by $M^{\sigma^n}$.\vspace{0.5em}

\noindent\textbf{Acknowledgments.} It is a pleasure to thank Brian Conrad, Dragos Ghioca, Joe Rabinoff, David Rydh, and Joe Silverman for helpful discussions.

\section{Constructing a $p$-adic family of modules}
\label{sec:p-adic-family}
As outlined in the introduction, the first step in proving Theorem \ref{thm: main2} is to construct a $p$-adic family of modules interpolating between the iterates $(\sigma^n)^*M$. That is our goal in this section.

\begin{notation}
Throughout this section and the subsequent two sections we use the following notation and assumptions.
\begin{enumerate}
\item we let $p$ be a prime number and we let $K$ be a field that is complete with respect to a non-Archimedean absolute value $|\,\cdot \, |$ such that $|p|=1/p$; 
\item we let $S=K\<x_i,\dots,x_d\>$ be the ring of convergent power series; 
\item we let $\sigma:S\to S$ be a $K$-algebra automorphism satisfying $|\sigma(x_i)-x_i|<p^{-c}$ for some $c>1/(p-1)$ for $i=1,\ldots ,d$;
\item we let $\mathfrak{o}$ denote the valuation ring of $K$.
\end{enumerate}
\end{notation}
In particular, $\sigma$ restricts to an $\mathfrak{o}$-algebra automorphism of $\mathfrak{o}\langle x_1,\ldots ,x_d\rangle$.

The following result is immediate from Theorem 1 and Remark 3 of \cite{bjorn}. We include the proof here since we make use of the notation in Lemma \ref{l:sigz-properties}.
\begin{proposition}[{\cite[Theorem 1]{bjorn}}]
\label{prop:affine-Bell-Poonen}
If $c > \frac{1}{1-p}$ and $\sigma(b)\equiv b\ (\bmod\ p^c)$ for all $b\in  \mathfrak{o}\langle x_1,\ldots ,x_d\rangle$, then there exists a map $\sigma^z:  S\to   S\< z\>$ such that $\sigma^z(b)|_{z=n}=\sigma^n(b)$ for all $n\in\mathbb{Z}$.
\end{proposition}
\begin{proof}
Let $\Delta:S\to S$ be defined by $\Delta(b):=\sigma(b)-b$. Let $\sigma^z:S\to S\<z\>$ be defined by
\[
\sigma^z(b) = \sum_{m\geq 0} {z \choose m}\Delta^m(b),
\]
where ${z \choose m}:= \frac{1}{m!}z(z-1)\dots(z-m+1)$. We must check that $\sigma^z$ is well-defined; that is, $|\frac{1}{m!}\Delta^m(b)|_p\to0$ as $m\to\infty$. Since $|\sigma(x_i)- x_i|\le p^{-c}$ and $\Delta$ is $\mathbb{Z}[p^{-c}]$-linear, we see
\[
|\Delta^m(b)| \leq p^{-mc} |b| <  |m!|,
\]
for $m$ large, which proves $\sigma^z$ is well-defined.

Lastly, we note that
\[
\sigma^z(b)|_{z=n}=\sum_{m=0}^n {n \choose m}\Delta^m(b)=(\Delta + \id)^n(b) = \sigma^n(b),
\]
which proves the result.
\end{proof}

We next define the ``$p$-adic powers'' of $\sigma$:
\begin{definition}
\label{def:sigma^b}
For every element $a\in\ZZ_p$ (and more generally, for every power bounded element $a\in  S$), we have a surjective map $\pi_a: S\<z\> \to  S$ defined by $z\mapsto a$. We define $\sigma^a: S\to S$ to be $\sigma^a=\pi_a\circ\sigma^z$.
\end{definition}

\begin{lemma}
\label{l:sigz-properties}
Under the hypotheses of Proposition \ref{prop:affine-Bell-Poonen}, the maps $\sigma^z$ and $\sigma^a$ satisfy the following properties:
\begin{enumerate}
\item\label{item:Z_p-alg-hom} $\sigma^z: S\to  S\<z\>$ is an injective homomorphism of $K$-algebras;
\item\label{item:composition} $\sigma^{a+a'} = \sigma^{a'}\circ\sigma^{a}$ for all $a,a'\in\ZZ_p$;
\item\label{item:aut} $\sigma^a$ is a $K$-algebra automorphism of $ S$ for all $a\in\ZZ_p$.
\end{enumerate}
\end{lemma}
\begin{proof}
We first show that $\sigma^z$ is additive. Note that for all $b,b'\in S$ and $n\in\NN$, we have $\sigma^z(b+b')|_{z=n}=\sigma^n(b+b')=\sigma^n(b)+\sigma^n(b')=(\sigma^z(b)+\sigma^z(b'))|_{z=n}$. Since $\sigma^z(b)+\sigma^z(b')-\sigma^z(b+b')\in S\<z\>$ and has roots at every element of $\NN$, by Strassman's Theorem (see \cite{strassman} or \cite[Theorem 4.1, p. 62]{cassels}) we know that $\sigma^z(b+b') = \sigma^z(b)+\sigma^z(b')$. 

Similarly, we see that $\sigma^z$ is multiplicative, and so it is a ring homomorphism. It is clear from the definition that $\sigma^z$ is $K$-linear. To show it is injective, note that if $\sigma^z(b)=0$, then $\sigma(b)=\sigma^z(b)|_{z=1}=0$, and so $b=0$. This proves (\ref{item:Z_p-alg-hom}).

Next, we recall the Chu-Vandermonde identity, which states that ${a+a' \choose m} = \sum_{i=0}^m{a \choose i}{a' \choose m-i}$ for all $a,a'\in\NN$. Since $\NN$ is dense in $\ZZ_p$, we see that the same identity holds when $a,a'\in\ZZ_p$. Since
\[
\sigma^{a+a'}(b) = \sum_{m\geq0}{a+a' \choose m}\Delta^m(b)
\]
and
\[
\sigma^{a'}(\sigma^{a}(b)) = \sum_{i,j\geq0}{a \choose i}{a' \choose j}\Delta^{i+j}(b) = \sum_{m\geq0}\sum_{i=0}^m{a \choose i}{a' \choose m-i}\Delta^{m}(b).
\]
this proves (\ref{item:composition}).

Property (\ref{item:aut}) follows from (\ref{item:composition}) since $\sigma^0 = \id$ and $\sigma^{a}\circ\sigma^{-a} = \id = \sigma^{-a}\circ\sigma^{a}$.
\end{proof}

Lastly, we define our $p$-adic family of modules:
\begin{definition}
\label{def:p-adic-family-ideals}
We use the notation and assumptions given in (1)--(4) at the beginning of this section. 
If $M$ is a finitely generated $S$-module and if we take a presentation $M=S^d/\langle r_1,\ldots ,r_m\rangle$, with $r_1,\ldots ,r_m\in S^d$, then we let $\M(z):=S\<z\>^d/\langle \sigma^z(r_1),\ldots ,\sigma^z(r_m)\rangle$. For each $a\in\ZZ_p$ (or more generally any 
power bounded element $a\in S$), we then define
\[
\M(a):=S^d/\langle \sigma^a(r_1),\ldots ,\sigma^a(r_m)\rangle,
\]
which is simply the image $\pi_a(\M(z))$ under the evaluation map sending $z\mapsto a$; i.e., the map $\M(z)\otimes_{S\langle z\rangle} S\langle z\rangle \to \M(z)\otimes_{S\langle z\rangle} S$ in which $S\langle z\rangle\to S$ is given by specializing at $z=a$.  Since $\sigma^z(M)|_{z=n}=(\sigma^n)^*(M)$, we can think of $\M(z)$ as a $p$-adic family  interpolating between the $(\sigma^n)^*(M)$ for $n\in\mathbb{Z}$.
\end{definition}

\section{Specializing complexes of modules}
\label{subsec:specialization-arg}
In this section, we gather several technical results concerning exactness of sequences of $S\<z\>$-modules after specializing $z$ at values of $\ZZ_p$. We recall that a module $M$ over a ring $R$ is \emph{saturated} with respect to an element $r\in R$ if for all $x\in M$, if $rx=0$ then $x=0$. 
We say $M$ is saturated with respect to a subset of $R$ if it is saturated with respect to every element in the subset.

\begin{lemma}
\label{l:sat-primary-decomp-general}
Let $R$ be a noetherian integral domain, let $\Sigma \subseteq R$ be an additive subgroup, and let $\Sigma_0:=\Sigma\setminus 0$. Fix an element $r\in R$. If $\M$ is a finitely generated $R$-module which is saturated with respect to $\Sigma_0$, then $\M$ is saturated with respect to $r-s$ for a cofinite set of $s\in \Sigma$. More specifically, the size of the set
\[
\{s\in \Sigma \mid \M\textrm{\ not\ saturated\ with\ respect\ to\ }r-s\}
\]
is at most the number of modules occurring in the primary decomposition of $(0)\subseteq\M$.
\end{lemma}
\begin{proof}
Since $R$ is noetherian and $\M$ is finitely generated, we have $(0)=\N_1\cap \cdots \cap \N_t$ where the $\N_i$ are primary submodules of $\M$; that is, for each $i$, if $x\in R$ is a zero divisor on $\M/\N_i$ then there is some $n\ge 1$ such that $x^n \M\subseteq \N_i$.  By taking a minimal decomposition, we may assume that each $\N_i$ is a proper submodule of $\M$ and that any proper intersection of $\N_1,\ldots ,\N_t$ is non-trivial.

If $\M$ is not saturated with respect to $r-s$, then there exists nonzero $f\in \M$ such that $(r-s)f=0$. Since $f$ is nonzero, there exists $i$ such that $f\notin\N_i$, and so $(r-s)^c\M \subseteq\N_i$ for some $c\geq 1$. 

Now, if there are distinct elements $s_1,\dots,s_m\in \Sigma_0$ with $m>t$ and $\M$ not saturated with respect to the $r-s_j$, then we can find $i$ and $j\neq k$ and a positive integer $c$ such that both $(r-s_j)^c \M$ and $(r-s_k)^c \M$ are contained in $\N_i$.  Now $\M\not\subseteq\N_i$ and so we can find a submodule $\M'\subseteq\M$ of the form $\M'=(r-s_j)^a (r-s_k)^b \M$, for some $a,b\ge 0$, with the property that $\M'\not\subseteq\N_i$ but $(r-s_j)\M'$ and $(r-s_k)\M'$ are both contained in $\N_i$.  But now it follows that $(s_j-s_k)\M'\subseteq \N_i$.  Since $\M'\not\subseteq \N_i$, we see $s_j-s_k$ is a zero divisor on $\M'/\N_i$, and so there is some natural number $d$ such that $(s_j-s_k)^d \M \subseteq \N_i$.  But now if we take nonzero $x\in \bigcap_{\ell\neq i} \N_\ell$ then $(s_j-s_k)^d x$ is in every $\N_\ell$ and hence $(s_j-s_k)^d x =0$.  By our saturation hypothesis, we see that $x=0$, a contradiction.  The result follows.
\end{proof}

Applying this to $R=A\<z\>$, $r=z$, and $\Sigma=\ZZ_p$ we have
\begin{corollary}
\label{cor:sat-primary-decomp-Tate}
Let $A$ be a noetherian $\ZZ_p$-algebra which is an integral domain, and let $\M$ be a finitely generated $A\<z\>$-module. If $\M$ is saturated with respect to $\ZZ_p\setminus 0$, then $\M$ is saturated with respect to $z-a$ for a cofinite set of $a\in\ZZ_p$.
\end{corollary}

In the following lemma, we discuss how kernels and images of $S\<z\>$-module maps interact with specialization at $z=a$.

\begin{lemma}
\label{l:coh-and-specialization}
Let $K$ be a field over $\QQ_p$. Let $S=K\<x_1,\dots,x_d\>$ and $g(z)\colon\M(z)\to\M'(z)$ be a morphism of finitely generated $S\langle z\rangle$-modules. Then $\im(g(a))=\im(g(z))|_{z=a}$ for all $a\in\ZZ_p$, and $\ker(g(a))=\ker(g(z))|_{z=a}$ for all but finitely many $a\in\ZZ_p$.
\end{lemma}
\begin{proof}
We begin by addressing the statement about kernels. Since $K$ is a field containing $\QQ_p$, every $S\langle z\rangle$-module is automatically saturated with respect to $\ZZ_p\setminus 0$. So, by Corollary \ref{cor:sat-primary-decomp-Tate} there are only finitely many $a\in \mathbb{Z}_p$ for which the finitely generated $S \ang{z}$-module $\M'(z)/{\rm Im}(g(z))$ is not saturated with respect to $z-a$. Assuming that $a$ is not contained in this finite set, we show that the conclusion of the lemma holds.

Let $\theta\in \M(a)$ be in the kernel of $g(a)$. Since the map $\M(z)\to \M(a)$ is surjective there is some $\theta(z)\in \M(z)$ with the property that $\theta=\theta(a)$.  Then by assumption $g(z)(\theta(z)) \in (z-a)\M(z)={\rm ker}(\M(z)\to \M(a))$.  Thus we have $g(z)(\theta(z))=(z-a)u(z)$ for some $u(z)\in \M(z)$. By our choice of $a$, we see that $u(z)\in {\rm Im}(g(z))$ and so $u(z)=g(z)(\theta'(z))$.  Thus $\theta(z)-(z-a)\theta'(z)$ is in the kernel of $g(z)$. Since $\theta(z)-(z-a)\theta'(z)$ is a lift of $\theta$, we have shown $\ker(g(z))\to\ker(g(a))$ is surjective, as desired.

Lastly, we show that $\im(g(z))|_{z=a}=\im(g(a))$ for all $a\in\ZZ_p$. First observe that if $\theta'(z)=g(z)(\theta(z))$, then after specialization we have $\theta'(a)=g(a)(\theta(a))$; thus, $\im(g(z))|_{z=a}\subseteq \im(g(a))$.  On the other hand, if $\theta'=g(a)(\theta)$ for some $\theta\in \M(a)$ then letting $\theta(z)\in \M(z)$ with $\theta(a)=\theta$ and taking $\xi(z)=g(z)(\theta(z))$, we see $\xi(z)\in \im(g(z))$ and $\xi(a)=\theta$. This gives the other containment.
\end{proof}

As an immediate consequence of the above lemma, we have

\begin{corollary}
\label{cor:exact->exact}
Let $K$ be a field over $\QQ_p$. Let $S=K\<x_1,\dots,x_d\>$, let $$\M''(z)\stackrel{f(z)}{\longrightarrow} \M(z)\stackrel{g(z)}{\longrightarrow} \M'(z)$$ be a sequence of finitely generated $S\langle z\rangle$-modules, and let $\HH(z)$ denote the cohomology module $\ker(g(z))/\im(f(z))$. Consider the induced sequence of $S$-modules
$$\M''(a) \stackrel{f(a)}{\longrightarrow} \M(a)\stackrel{g(a)}{\longrightarrow} \M'(a)$$
and denote its cohomology by $\HH(a)$. Then for all but finitely many $a\in\ZZ_p$,
\[
\HH(a)=\HH(z)|_{z=a}.
\]
In particular, if the former sequence is exact, then the latter is for all but finitely many $a\in\ZZ_p$.
\end{corollary}

\section{Proof of Theorems \ref{thm: main2} and \ref{thm: main3}}
\label{sec:proof-main-thm2}
In this section, we prove the following result, which clearly implies Theorem \ref{thm: main3}, and hence Theorem \ref{thm: main2}.
\begin{theorem}
\label{thm: main3-more-precise}
Adopt the assumptions of Theorem \ref{thm: main2}, let $\M(z)$ be as in Definition \ref{def:p-adic-family-ideals}, and let $\N(z)=S\<z\>\otimes_S N$. If $\Tor_i^{S\ang{z}}(\M(z),\N(z))=0$, then $\Tor_i^S(\M(a),N)=0$ for all but finitely many $a\in\ZZ_p$. If $\Tor_i^{S\ang{z}}(\M(z),\N(z))\neq0$, then the set of $a\in\ZZ$ for which $\Tor_i^S(\M(a),N)\neq0$ is, up to addition and removal of a finite set, a finite union of arithmetic progressions of difference $p^r$ for some $r\geq0$.
\end{theorem}

We prove the theorem after giving a preliminary result relating $\Tor_i^{S}(\M(a),\N(a))$ to the specialization of $\Tor_i^{S\ang{z}}(\M(z),\N(z))$.

\begin{proposition}
Let $K$ be a field over $\QQ_p$ and let $S=K\<x_1,\dots,x_d\>$. If $\M(z)$ and $\N(z)$ are finitely generated $S\langle z\rangle$-modules, then
\[
\Tor_i^S(\M(a),\N(a))=\Tor_i^{S\ang{z}}(\M(z),\N(z))|_{z=a}
\]
for all but finitely many $a\in\ZZ_p$.
\label{prop: half2_now-more-general}
\end{proposition}
\begin{proof}
By \cite[Proposition 6.5]{F-isocrystals}, 
there is a finite free resolution
\[
0\to \Proj_d(z)\to \Proj_{d-1}(z) \to\cdots \to\Proj_0(z)\to \N(z)\to 0.
\]
Then Corollary \ref{cor:exact->exact} shows that
\[
0\to \Proj_d(a)\to\Proj_{d-1}(a) \to\cdots \to\Proj_0(a)\to \N(a)\to 0
\]
is a finite free resolution of $\N(a)$ for all $a\in\ZZ_p\setminus\T$, where $\T$ is a finite set.

Tensoring the former resolution with $\M(z)$, we obtain a complex
$$0\to \Proj_d(z)\otimes_{S\ang{z}} \M(z) \to \Proj_{d-1}(z)\otimes_{S\ang{z}} \M(z) \to \cdots \to \Proj_1(z)\otimes_{S\ang{z}} \M(z) \to \Proj_0(z)\otimes_{S\ang{z}} \M(z)\to0.$$
Notice that this complex specializes to
$$0\to \Proj_d(a)\otimes_S \M(a) \to \Proj_{d-1}(a)\otimes_S \M(a) \to \cdots \to \Proj_1(a)\otimes_S \M(a) \to \Proj_0(a)\otimes_S \M(a)\to0$$ 
and so another application of Corollary \ref{cor:exact->exact} shows that for all $a$ outside of a finite set $\T'$, the cohomology of the former complex, namely $\Tor_i^{S\ang{z}}(\M(z),\N(z))$, specializes to cohomology of the latter complex, which we temporarily denote by $\HH_i(a)$.

Now, for all $a\notin\T$, the complex $\Proj_\bullet(a)$ is exact, and so $\HH_i(a)=\Tor_i^{S}(\M(a),\N(a))$. Thus, $\Tor_i^S(\M(a),\N(a))=\Tor_i^{S\ang{z}}(\M(z),\N(z))|_{z=a}$ for all $a\notin\T\cup\T'$.
 \end{proof}

\begin{proof}[{Proof of Theorem \ref{thm: main3-more-precise}}]
By Proposition \ref{prop: half2_now-more-general}, we know
\[
\Tor_i^S(\M(a),N)=\Tor_i^{S\ang{z}}(\M(z),\N(z))|_{z=a}
\]
for all but finitely many $a\in\ZZ_p$. So, if $\Tor_i^{S\ang{z}}(\M(z),\N(z))=0$, then $\Tor_i^S(\M(a),N)=0$ with finitely many exceptions. If $\Tor_i^{S\ang{z}}(\M(z),\N(z))$ is a non-zero module, then applying our module-theoretic Strassman Theorem \ref{theorem: strassman}, we see that after addition or removal of a finite set, the $a\in\ZZ_p$ where $\Tor_i^S(\M(a),N)\neq0$ is a finite union of arithmetic progressions with difference $p^r$ for some $r\geq0$.
\end{proof}

\begin{corollary}
\label{cor: application}
Let $k$ be a field of characteristic zero and let $R$ be a $k$-algebra that is a regular noetherian local ring such that the field of fractions of $R$ is a finitely generated extension of $k$ and such that $R/\mathfrak{m}=k$.  Suppose that $\sigma: R\to R$ is a $k$-algebra automorphism of R and that $M$ and $N$ are finitely generated $R$-modules.  Then for each $i\ge 1$ we have that $$\{n\colon {\rm Tor}^R_i(M^{\sigma^n},N)\neq 0\}$$ is a finite union of arithmetic progressions up to addition and removal of finite sets. 
\end{corollary}
\begin{proof} Let $d$ be the Krull dimension of $R$.  Pick $t_1,\ldots , t_d\in \mathfrak{m}$ that generate the maximal ideal. By Cohen's structure theorem, we have that the completion, $\widehat{R}$, of $R$ is isomorphic to the power series ring $k[[t_1,\ldots ,t_d]]$. Since $\sigma(\mathfrak{m})=\mathfrak{m}$, $\sigma$ extends to an automorphism of $k[[t_1,\ldots ,t_d]]$. Since the field of fractions of $R$ is finitely generated as an extension of $k$, we have that 
${\rm Frac}(R)$ is finite over the subfield $k(t_1,\ldots ,t_d)$.

 Hence $\sigma(t_i)=f_i(t_1,\ldots ,t_d)$ where each $f_i(t_1,\ldots ,t_d)$ is algebraic over $k(t_1,\ldots ,t_d)$. Now an algebraic power series has the property that its coefficients lie in a finitely generated $\mathbb{Z}$-algebra. (This follows from a general result of Denef and Lipshitz \cite{DL}, which shows in particular that the set of coefficients of an algebraic power series in $d$ variables is a subset of the collection of coefficients of some rational power series in $2d$ variables.)  Thus there is a finitely generated $\mathbb{Z}$-subalgebra $A$ of $k$ such that $f_1,\ldots ,f_d\in A[[x_1,\ldots ,x_d]]$.  

Since Tor commutes with completion for finitely presented modules we may work with $\widehat{R}$ and replace $M$ by $M\otimes_R \widehat{R}$ and $N$ by $N\otimes_R \widehat{R}$.  Now let $J$ denote the Jacobian of $(f_1,\ldots ,f_d)$ at the origin. Then the determinant of $J$ is a nonzero element of $A$.  By adjoining the inverse of $\det(J)$ to $A$, we may assume that $J$ is invertible in $A$.  By construction, $\sigma$ restricts to an automorphism of $A[[x_1,\ldots ,x_d]]$. Furthermore, we have that $M\cong \widehat{R}^t/L$ and $N\cong \widehat{R}^s/E$ for submodules $L$ and $E$ of $\widehat{R}^t$ and $\widehat{R}^s$ respectively.  Then by taking generators for $L$ and $E$ and adjoining the coordinates of all elements in these generating sets to $A$, we may assume that $A$ is still a finitely generated $\mathbb{Z}$-algebra and that there are finitely presented $A[[x_1,\ldots ,x_d]$-modules $M_0$ and $N_0$ such that $M\cong M_0\otimes_{A} k$ and $N\cong N_0\otimes_{A} k$.  Observe that $k[[x_1,\ldots ,x_d]]$ is flat over $A[[x_1,\ldots ,x_d]]$, as it is a free module over a localization of $A[[x_1,\ldots ,x_d]]$.  Let $k_0$ denote the field of fractions of $A$.  Then we have \cite[Lemma 10.75.1]{Stacks} 
$${\rm Tor}_i^{k[[x_1,\ldots ,x_d]]}(M, N)\cong
{\rm Tor}_i^{A[[x_1,\ldots ,x_d]]}(M_0,N_0)\otimes_{A[[x_1,\ldots ,x_d]]} k[[x_1,\ldots ,x_d]].$$ 

 Moreover, since 
$k[[x_1,\ldots ,x_d]]$ is faithfully flat over $k_0[[x_1,\ldots ,x_d]]$, we see that 
$${\rm Tor}_i^{k[[x_1,\ldots ,x_d]]}( M^{\sigma^n}, N)=0$$ if and only if 
$${\rm Tor}_i^{k_0[[x_1,\ldots ,x_d]]}( M_0^{\sigma^n}\otimes_A  k_0, N_0\otimes_A k_0)\neq 0.$$

Now let $M_1=M_0\otimes_A k_0$ and let $N_1=N_0\otimes_A k_0$.  Then pick a maximal ideal $Q$ of $A$ such that $A_Q$ is regular.  Since $A$ is a finitely generated $\mathbb{Z}$-algebra, $A/Q$ is a finite field.  We take the completion, $\widehat{A_Q}$, of $A_Q$.  Then by Cohen's structure theorem \cite[Section 29]{Mat}, this is a power series ring in a finite number of variables over a finite unramified extension $\mathfrak{o}$ of $\mathbb{Z}_p$ for some prime $p$.  Let $F$ denote the field of fractions of this completion of $A_Q$. Then since $\sigma$ is the identity on $k_0$ we may extend the automorphism $\sigma$ of $k_0[[x_1,\ldots ,x_d]]$ to an automorphism of $k_0[[x_1,\ldots ,x_d]]\otimes_{k_0} F$.  By then taking limits we may extend this to an automorphism of $F[[x_1,\ldots ,x_d]]$.  

Since $F[[x_1,\ldots ,x_d]]$ is a faithfully flat extension of $k_0[[x_1,\ldots ,x_d]]$, we see that
$${\rm Tor}_i^{k[[x_1,\ldots ,x_d]]}( M^{\sigma^n}, N)=0$$ if and only if
$${\rm Tor}_i^{F[[x_1,\ldots ,x_d]]}( (M')^{\sigma^n}, N')=0,$$ where $M'=M_1\otimes_{k_0[[x_1,\ldots ,x_d]]} F[[x_1,\ldots ,x_d]$ and $N'=N_1\otimes_{k_0[[x_1,\ldots ,x_d]]} F[[x_1,\ldots ,x_d]$.  Moreover, by construction the induced automorphism $\sigma$ of $F[[x_1,\ldots ,x_d]]$ has the property that $\sigma(x_i)=f_i(x_1,\ldots ,x_d)\in (\widehat{A_Q})[[x_1,\ldots ,x_d]]\subseteq F[[x_1,\ldots ,x_d]]$ has Gauss norm $\le 1$ and the determinant of the Jacobian at the origin has $p$-adic norm exactly one, since it is a unit in $A$ by how we defined $A$.  In particular, since $\widehat{A_Q}$ has finite residue field, we see that some iterate, $\sigma^m$, of $\sigma$ has the property that its Jacobian at the origin is congruent to the identity modulo the maximal ideal of $\widehat{A_Q}$. 

Now to finish the proof off,  let $u_i=x_i/p$ for $i=1,\ldots ,d$.  Then $F[[x_1,\ldots ,x_d]]=F[[u_1,\ldots ,u_d]]$.  Observe that $$\sigma(u_i)=\frac{1}{p}\cdot f_i(pu_1,\ldots ,pu_d) \equiv L_i(u_1,\ldots ,u_d)~(\bmod\ p\mathfrak{o}),$$ where $L_i$ is the linear part of $f_i$.  Then by the above remarks we see that $\sigma^m(u_i)\equiv u_i~(\bmod\ p\mathfrak{o})$ for $i=1,\ldots ,d$.  Then by Theorem \ref{thm: main2} we see that 
for $j=0,\ldots ,m-1$ we have $$\{n\colon {\rm Tor}_i^{F[[u_1,\ldots ,u_d]]}(\sigma^{mn+j}M',N')\neq 0\}$$ is, up to addition and removal of a finite set, equal to a finite union of arithmetic progressions.  The result follows.
\end{proof}

\section{Theorem \ref{thm: main1} for surfaces}
\label{sec:surface-case}
In this section we prove the surface case of Theorem~\ref{thm: main1}.  We first give some preliminary results. 

\begin{lemma} Let $(R,\mathfrak{m})$ be a regular local ring of dimension two and let $M$ be a finitely generated $R$-module whose support is $\{P_1,\ldots ,P_d\}$ with each $P_i$ of height at most one.  Then ${\rm projdim}(M)\le 1$.
\label{lem: curve}
\end{lemma}
\begin{proof}
If it is not, then ${\rm projdim}(M)\ge 2$ and so ${\rm Tor}_2^R(R/\mathfrak{m},M)\neq 0$.  Thus it suffices to show that
${\rm Tor}_2^R(R/\mathfrak{m},M)= 0$.
Let $x$ and $y$ be generators for $\mathfrak{m}$.  Then we have a resolution 
$$0\to R\to R^2\to R\to R/\mathfrak{m},$$ where the map $R\to R^2$ is the map given by $1\mapsto (y,-x)$ and the map from $R^2\to R$ is the map $(a,b)\mapsto xa+yb$.  Then from this resolution we see that 
${\rm Tor}_2^R(R/\mathfrak{m},M)$ is just the kernel of the map $M=M\otimes_R R\to M^2=M\otimes_R R^2$ given by
$m\mapsto (ym,-xm)$.  Since the maximal ideal of $R$ is not in the support of $M$, we see that the kernel is trivial.  
\end{proof}
Notice this lemma shows that a torsion-free coherent sheaf $\mathcal{T}$ on a smooth surface $X$ has the property that $\shTor_2^X(\mathcal{T},\mathcal{N})=0$ for any coherent sheaf $\mathcal{N}$.
\begin{proposition} Let $k$ be an algebraically closed field, let $X$ be a smooth surface over $k$, and let $\N$ and $\T$ be coherent sheaves on $X$ that each have the property that they are supported on a finite set of points and a finite set of curves.  Then $\shTor_1^X(\T,\N)$ is nonzero if and only if either some irreducible subvariety in the  support of $\T$ contains an irreducible subvariety in the support of $\N$ or some irreducible subvariety in the support of $\N$ contains an irreducible subvariety in the support of $\T$.  \label{prop: lichtenbaum}
\end{proposition}
\begin{proof}
Let ${\rm Supp}(\T) = \{V_1,\ldots ,V_m\}$ and ${\rm Supp}(\N)=\{W_1,\ldots ,W_q\}$, where each $V_i$ and each $W_j$ is either a point or an irreducible curve.  Now suppose that either $V_i\subseteq W_j$ for some $i,j$ or $W_i\subseteq V_j$ for some $i,j$.  Then after switching $\T$ and $\N$, if necessary and reindexing if necessary, we may assume that $V_1\subseteq W_1$.  We shall show $\shTor_1^X(\T,\N)$ is nonzero.  To see this, suppose that $\shTor_1^X(\T,\N)=0$.  Then since ${\rm Tor}$ commutes with localization we have
${\rm Tor}_1^R(\T_{V_1},\N_{V_1})=0$, where $R=\mathcal{O}_{X,V_1}$ is a regular local ring. We let $P$ denote the maximal ideal of $R$. Since $V_1$ is in the support of $\T$ we have that there is an element in $\T_{V_1}$ that is annihilated by $P$.  Hence the depth of $\T_{V_1}$ is zero.  By a result of Lichtenbaum \cite[Corollary 6]{Licht} we then have that the homological dimension of $\N_{V_1}$ is zero; that is, ${\rm Tor}_1^R(R/P, \N_{V_1})=0$. But this implies that $\N_{V_1}$ is torsion-free, which is impossible since $W_1\supseteq V_1$ is in the support of $\N$.  Thus we see that $\shTor^X_1(\T,\N)$ is nonzero, completing half of the proof.

Suppose that no $V_i$ contains a $W_j$ and that no $W_i$ contains a $V_j$.  We claim that $\shTor_1^X(\T,\N)=0$.  To see this, suppose that $\shTor_1^X(\T,\N)\neq 0$.  Then there is some $p\in X$ such that 
${\rm Tor}_1^R(\T_p,\N_p)\neq 0$, where $R=\mathcal{O}_{X,p}$.  Since ${\rm Tor}_1^R(\T_p,\N_p)\neq 0$ there is some $i$ and some $j$ such that $p\in V_i\cap W_j$. Let $\mathfrak{m}$ denote the maximal ideal of $R$. We note that $\mathfrak{m}$ cannot be an element of the support of $\T_p$ since we would then have $p=V_k$ for some $k$ and so $V_k\subseteq W_j$, which we have assumed not to be the case; similarly, $\mathfrak{m}$ cannot be an element of the support of $\N_p$.  In particular, the support of $\T_p$ consists of a finite set of height one primes and the same holds for the support of $\N_p$.  Moreover, by assumption the support of $\T_p$ and the support of $\N_p$ cannot share a common height one prime ideal. Thus, $\mathfrak{m}$ must be in the support of ${\rm Tor}_1^R(\T_p,\N_p)$, since the supports of $\T_p$ and $\N_p$ are both unions of height one primes and they do not share any common primes.  Then a result of Auslander's \cite[Theorem 2]{supp-Tor} gives that ${\rm projdim}(\T_p)+{\rm projdim}(\N_p)\ge 3$.  But Lemma \ref{lem: curve} gives that 
${\rm projdim}(\T_p)+{\rm projdim}(\N_p)\le 2$, a contradiction.  The result follows.
\end{proof}

The proof of Theorem~\ref{thm: main1} for surfaces is a series of reductions.
We first show that the result holds if both sheaves are supported on proper subvarieties of an affine $X$.
We then prove a lemma allowing us to restrict from a quasiprojective variety to an open affine subset.
Finally, in Theorem~\ref{thm:nolabel} we put the pieces together, using Corollary~\ref{cor: application} to handle behaviour at points of finite $\sigma$-order.

\begin{proposition}
\label{prop: supportonedim}
Let $k$ be an algebraically closed field of characteristic zero, let $X$ be a quasiprojective irreducible surface over $k$, and let $\sigma: X\to X$ be an automorphism of $X$.  If $\mathcal{M}$ and $\mathcal{N}$ are coherent sheaves on $X$ whose supports each have dimension at most one and $U$ is a non-empty affine open subset of $X$ then for each $i\ge 1$, the set of natural numbers $n$ for which
$${\rm Tor}^{\mathcal{O}_X(U)}_i(\mathcal{M}^{\sigma^n}(U),\mathcal{N}(U))\neq (0)$$ is a finite union of infinite arithmetic progressions up to addition and removal of finite sets.
\end{proposition}

\begin{proof}
We first prove the result when $i=1$. Write ${\rm Supp}(\M) = \{V_1,\ldots ,V_m\}$ and ${\rm Supp}(\N)=\{W_1,\ldots ,W_q\}$, where each $V_i$ and each $W_j$ is either a point or an irreducible curve.  Then since $\supp\M^{\sigma^n} = \{\sigma^{-n}(V_1),\ldots ,\sigma^{-n}(V_q)\}$, by Proposition \ref{prop: lichtenbaum}, we have that
\[\Tor_1^{\mathcal{O}_X(U)}(\M^{\sigma^n}(U),\N(U))\neq 0\]
 if and only if $\sigma^{-n}(V_i)\subseteq W_j$ or $\sigma^{-n}(V_i)\supseteq W_j$ for some $i$ and $j$ with $\sigma^{-n}(V_i)\cap U$ and $W_j\cap U$ non-empty.  For each $i\le m$ and $j\le q$ such that $V_i\cap U$ and $W_j\cap U$ are non-empty we let
$\mathcal{S}(i,j)$ denote the set of integers $n$ for which $\sigma^{-n}(V_i)\subseteq W_j$ and we let
$\mathcal{S}'(i,j)$ denote the set of integers $n$ for which $\sigma^{-n}(V_i)\supseteq W_j$.  Then
$$\{n\in \mathbb{Z}\colon \Tor_1^{\mathcal{O}_X(U)}(\M^{\sigma^n}(U),\N(U))\neq 0\} = \bigcup_{i=1}^m \bigcup_{j=1}^q  \left(\mathcal{S}(i,j)\cup \mathcal{S}'(i,j)\right).$$ 
Thus it is sufficient to show that for each $i$ and $j$ both $\mathcal{S}(i,j)$ and $\mathcal{S}'(i,j)$ are both a finite union of complete doubly infinite arithmetic progressions up to addition and removal of finite sets.  By symmetry it is enough to just show this for $\mathcal{S}(i,j)$.  Since $V_i$ is either a point or a curve and $W_j$ is either a point or a curve, there are four cases to consider. If $V_i$ is a curve and $W_j$ is a point, then $\mathcal{S}(i,j)$ is empty.  If $V_i$ is a point and $W_i$ is a point then $\mathcal{S}(i,j)=\{n\colon \sigma^{-n}(V_i)=W_j\}$.  This is easily seen to either the empty set, a single integer, or a single arithmetic progression.  If $V_i$ is a curve and $W_i$ is a curve then $\mathcal{S}(i,j)=\{n\colon \sigma^{-n}(V_i)=W_j\}$, which is again either the empty set, a single integer, or a single arithmetic progression. Finally, if
$V_i$ is a point and $W_i$ is a curve then $\mathcal{S}(i,j)=\{n\colon \sigma^{-n}(V_i)\in W_j,~\sigma^{-n}(V_i)\in U\}$.  Since $U^c$ and $W_j$ are both Zariski closed, we see that this set is a finite union of arithmetic progressions up to addition and subtraction of finite sets (cf. \cite{BGT}).  Thus we have shown that
the set of $n$ for which $\shTor_1^X(\M^{\sigma^n},\N)\neq 0$ is a finite union of arithmetic progressions up to addition and removal of finite sets. 

We now quickly argue that for each $i\ge 1$, the set of $n$ for which ${\rm Tor}_i^{\mathcal{O}_X(U)}(\M^{\sigma^n}(U),\N(U))\neq 0$ is a finite union of arithmetic progressions up to addition and removal of finite sets.  We have just proven the case when $i=1$.  Since $X$ is a smooth surface, it remains only to prove the case when $i=2$.  
The set of $n$ for which ${\rm Tor}_1^{\mathcal{O}_X(U)}(\M^{\sigma^n}(U),\N(U))\neq 0$ is a finite union of arithmetic progressions up to addition and removal of finite sets. Then we have a short exact sequence
\[
0\to\M'\to\F\to\M\to0
\]
with $\F$ locally free and coherent and $\M'$ coherent. Since $\sigma^n(\F)$ is also locally free, we see ${\rm Tor}_i^{\mathcal{O}_X(U)}(\F^{\sigma^n}(U),\N(U))=0$ for all $i>0$. Then ${\rm Tor}_{1}^{\mathcal{O}_X(U)}((\M')^{\sigma^n}(U),\N(U))$ is isomorphic to ${\rm Tor}_{2}^{\mathcal{O}_X(U)}(\M^{\sigma^n}(U),\N(U))$, and so we obtain the desired result.
\end{proof}

\begin{lemma}
\label{lem: countable}
Let $k$ be an uncountable algebraically closed field, let $X$ be an irreducible quasiprojective variety over $k$, and let $F$ be a finite subset of $X$ and let $T$ be a countably infinite subset of $X$ with $F\cap T=\emptyset$. Then there is a rational function $f$ on $X$ with the following properties:
\begin{enumerate}
\item $f$ is regular at all points in $F\cup T$;
 \item $f(x)=0$ for $x\in F$;
 \item $f$ does not vanish at any point in $T$;
 \item $X\setminus V(f)$ is affine. 
 \end{enumerate}
\end{lemma}
\begin{proof}
We fix an embedding $X\to \mathbb{P}^n$ and take homogeneous coordinates $[x_0:\cdots ,x_n]$ for $\mathbb{P}^n$. We first claim that there is a homogeneous one-form $L:=c_0 x_0+\cdots +c_n x_n$ with $[c_0:\cdots :c_n]\in \mathbb{P}^n$ such that the zero locus of $L$ is disjoint from $F\cup T$.   To see this, observe that the collection of homogeneous one-forms can be identified with $\mathbb{P}^n(k)$.  Then the set of forms that vanish at a point $y\in \mathbb{P}^n$ is a proper closed subset of $\mathbb{P}^n(k)$. Since $F\cup T$ is countable and $\mathbb{P}^n(k)$ cannot be written as a countable union of proper subvarieties we see there is some homogeneous one form whose zero locus completely avoids $F\cup T$.  By changing variables, we may assume that $L=x_0$.  Now for each $y\in F$, we consider the collection of homogeneous one-forms that vanish at $y$. As before, this can be identified with $\mathbb{P}^{n-1}$ and as before, we see there is some form $L_y$ that vanishes at $y$ and whose zero locus avoids $T$.  Now $X-V(L_y)$ is a non-empty open affine neighbourhood of $X$.  We now let
$f=\prod_{y\in F} L_y/x_0$.  Then $f$ is a rational function on $X$ satisfying properties (1)--(4).
\end{proof}
We now reduce the general case to the preceding result by using Corollary \ref{cor: application}.
\begin{theorem}
Theorem \ref{thm: main1} holds for smooth surfaces.
\end{theorem}\label{thm:nolabel}
\begin{proof}
We first note that by replacing $X$ by $X\times_k k'$ for some uncountable algebraically closed extension $k'$ of $k$ we may assume that our base field $k$ is uncountable and algebraically closed.

Let $i\ge 1$ and let $p_1,\ldots ,p_d$ denote the $\sigma$-periodic points that are elements of $\supp\M\cup \supp\N$.  We may replace $\sigma$ by an iterate and assume that $p_1,\ldots ,p_d$ are fixed points of $\sigma$. Then for $j=1,\ldots ,d$, we let
$$\mathcal{X}_j = \{n\in \mathbb{Z}\colon {\rm Tor}_i^{\mathcal{O}_{X,p_j}}((\M^{\sigma^n})_{p_j},\N_{p_j})\neq 0\}.$$  By Corollary \ref{cor: application} we have that $\mathcal{X}_j$ is, up to addition and removal of a finite set, a finite union of arithmetic progressions.
We next let $$\mathcal{U}= \{n\in \mathbb{Z}\colon {\rm Tor}^{\mathcal{O}_{X,q}}_i((\M^{\sigma^n})_{q},\N_{q})\neq 0~{\rm for~some}~q\in X\setminus \{p_1,\ldots ,p_d\}\}.$$
Then the set of integers $n$ for which $\shTor_i^X(\M^{\sigma^n},\N)\neq 0$ is the (not necessarily disjoint) union of the $\mathcal{X}_j$ and $\mathcal{U}$, and so it is sufficient to show that $\mathcal{U}$ is, up to addition and removal of a finite set, a finite union of arithmetic progressions.
Now for each $n\in \mathcal{U}$ we can pick some point $q_n\in X\setminus \{p_1,\ldots ,p_d\}$ that witnesses the non-vanishing of $\shTor_i^X(\M^{\sigma^n},\N)$; we then let $T=\{\sigma^j(q_n)\colon j\in \mathbb{Z}, n\in \mathcal{U}$.  Then $T$ is a countable subset of $X$ that avoids $p_1,\ldots ,p_d$ and so by Lemma \ref{lem: countable}, there is a rational function $f$ that is regular at $p_1,\ldots, p_d$ and at all points in $T$ with the following properties:
\begin{enumerate}
 \item $f(p_i)=0$ for $i=1,\ldots ,d$;
 \item $f$ does not vanish at any point in $T$;\item $U:=X\setminus V(f)$ is affine. 
 \end{enumerate}
 We note that $\sigma$ need not induce an automorphism of $U$.  
 
 For $n\in \mathcal{U}$ we have that
${\rm Tor}^U_i(\M^{\sigma^n}(U),\N(U))$ does not vanish.
Now we have a short exact sequence of sheaves
$$0\to \N'\to \N\to \N''\to 0$$ with $\N'$ having support of dimension at most one and $\N''$ being torsion free.
By Lemma \ref{lem: torsion free}, we have that 
 ${\rm Tor}_j^U(\M^{\sigma^n}(U),\N''(U))=0$ for all but finitely many $n$ for all $j\ge 1$.  This then gives that
 $${\rm Tor}_i^U(\M^{\sigma^n}(U),\N(U))\cong {\rm Tor}_i^U(\M^{\sigma^n}(U),\N'(U))$$
 for all $i\geq1$ and all but finitely many $n$. Now we similarly have an exact sequence
 $$0\to \M'\to \M\to \M''\to 0$$ with $\M'$ having support of dimension at most one and $\M''$ being torsion free.
 Then by Lemma \ref{lem: torsion free} we have
 ${\rm Tor}_j^U((\M'')^{\sigma^n}(U),\N''(U))=0$ for all but finitely many $n$ for all $j\ge 1$ and so we get an isomorphism
  $${\rm Tor}_i^U(\M^{\sigma^n}(U),\N(U))\cong {\rm Tor}_i^U(\M^{\sigma^n}(U),\N'(U))\cong {\rm Tor}_i^U((\M')^{\sigma^n}(U),\N'(U)).$$
  By Proposition \ref{prop: supportonedim}, since $\dim\supp\M'$ and $\dim\supp \N'\le 1$ we then see that the set of $n$ for which ${\rm Tor}_i((\M')^{\sigma^n}(U),\N'(U))$ is, up to addition and removal of a finite set, a finite union of arithmetic progressions.  The result now follows.
\end{proof}

\begin{lemma}
\label{lem: torsion free}
Let $k$ be an algebraically closed field of characteristic zero and let $X$ be an irreducible smooth quasiprojective surface over $k$, let $U$ be a non-empty affine open subset of $X$, and let $\sigma\in {\rm Aut}_k(X)$. If $\F$ is a torsion-free coherent sheaf on $X$ and $\N$ is a coherent sheaf on $X$ such that there are no $\sigma$-periodic points in $X$ that are elements of the support of $\N$ then  
 ${\rm Tor}_j^U(\F^{\sigma^n}(U),\N(U))=0$ for all but finitely many $n$ for all $j\ge 1$ and
 ${\rm Tor}_j^U(\N^{\sigma^n}(U),\F(U))=0$ for all but finitely many $n$ for all $j\ge 1$. 
\end{lemma}
\begin{proof}
We do the case where ${\rm Tor}_j^U(\F^{\sigma^n}(U),\N(U))=0$, the other case following mutatis mutandis.

We have a short exact sequence
$$0\to \F\to \E\to \Q\to 0$$ with $\E$ coherent and locally free and $\Q$ coherent having support of dimension at most one.
Then since $\E$ is locally free, we have an isomorphism
 ${\rm Tor}_j^U(\F^{\sigma^n}(U),\N(U))\cong {\rm Tor}_{j+1}^U((\Q^{\sigma^n})(U),\N(U))$. Since $X$ is a smooth surface, we see that the only case that is not immediate is when $j=1$ and so it is sufficient to show that the set of $n$ such that
${\rm Tor}_{2}^U((\Q^{\sigma^n})(U),\N(U))$ is nonzero is finite. 
Let $q_1,\ldots ,q_r$ be the points in $U$ that are elements of $\supp\N$; i.e., the associated points of $\N$ in $U$. 
By Lemma \ref{lem: curve} we see that if 
${\rm Tor}_{2}^U((\Q^{\sigma^n})(U),\N(U))$ is nonzero, then there must be some  $q_i$ that is an element of the support of  $\Q^{\sigma^n}$ and such that
${\rm Tor}_{2}^U((\Q^{\sigma^n})_{q_i},\N_{q_i})\neq 0$.  Thus  $\sigma^{-n}(q_i)$ must be an element of $\supp\Q$.  But by assumption, $q_1,\ldots ,q_r$ are not $\sigma$-periodic points and so the collection of $n$ for which $\sigma^{-n}(q_i)$ is an element of $\supp\Q$ is a finite set.  

\end{proof}

\section{Theorem \ref{thm: main1} when $\sigma$ lies in an algebraic group}
\label{sec:Sue}

In this section we prove the remaining case of Theorem~\ref{thm: main1}.  

\begin{theorem}\label{thm1.3.2}
Let $k$ be an algebraically closed field of characteristic zero, let $X$ be a nonsingular quasiprojective variety over $k$, and let $G$ be an algebraic group contained in $ \aut_k(X)$.   
Let $\sigma \in G$ and let $\M$, $\N$ be coherent sheaves on $X$.
For all $i \geq 1$ the sets:
\[ 
\{ n \in \ZZ \ | \ \shTor_i^X((\sigma^n)^* \M, \N) = 0\}\]
and 
\[ 
\{ n \in \ZZ \ | \ \shTor_i^X((\sigma^n)^* \M, \N) \neq 0\}\]
are finite unions of infinite arithmetic progressions up to the addition and removal of finite sets.
\end{theorem}

Theorem~\ref{thm1.3.2} follows from the following result.
\begin{proposition}\label{prop:genericbehaviour}
Let  $k$ be an algebraically closed field of characteristic zero, let $X$ be a nonsingular quasiprojective variety over $k$, and let $H$ be a connected component of  an algebraic group $G$ contained in $ \aut_k(X)$. 
Let $\M, \N$ be coherent sheaves on $X$.  
For all $i\geq1$, there is an open subset $V$ of $H$ so that either
\[ \shTor_i^X(g^\star \M, \N) = 0 \mbox{ for all } g \in V,\]
or
 \[ \shTor_i^X(g^\star \M, \N) \neq 0 \mbox{ for all } g \in V.\]
 \end{proposition}
 
 \begin{proof}
 Let
 \[ \xymatrix{ G \times X \ar[r]^{p} \ar[d]_{q} & X \\ G & }\]
 be the two projection maps, and let $\mu:  G \times X \to X$ be the map defining the action of $G$ on $X$.
 Consider coherent sheaves  $\F$ on $G\times X$ and $\E$ on $X$.  When we write $\F \otimes_X \E$, we implicitly assume that $X$ acts on $\F$ via $\mu$.  
 We note that $\mu$ is flat and so  we have
 \[ \shTor_i^{G \times X}(\F, \mu^* \E) \cong \shTor_i^X(\F, \E).\]

 Let $\L_{\bullet} \to \M$ be a locally free resolution of $\M$ (which is finite by assumption).
 Consider the complex
 \[ \C_\bullet = p^* \L_\bullet \otimes_X \N\]
 of sheaves on $G \times X$.
 The sheaves $p^* \L_\bullet$ are an $X$-flat resolution of $p^*\M$, and so the 
  $i$th homology  of $\C_\bullet$ is $\shTor_i^X(p^* \M, \N)$. 
 These homology groups are computed via exact sequences
 \[ 0 \to \Z_{i+1} \to \C_{i+1}\to \B_i \to 0\]
 and
 \[ 0 \to \B_i \to \Z_i \to   \shTor_i^X(p^* \M, \N) \to 0\]
 for $0 \leq i \leq \dim X $.
 By generic flatness, there is a dense open subset $V \subset H$ such that each of the finitely many  $\B_i $ and $\shTor_i^X(p^* \M, \N)$ are flat over $V$.
 Therefore, if $g \in V$, the  sequences
 \[ 0 \to \Z_{i+1}\otimes_G k_g \to(p^*\L_{i+1} \otimes_X \N) \otimes_G k_g\to \B_i\otimes_G k_g \to 0\]
 and
 \[ 0 \to \B_i \otimes_G k_g\to \Z_i\otimes_G k_g \to   \shTor_i^X(p^* \M, \N)\otimes_G k_g \to 0\]
 are still exact, and so 
\[ H_i(\C_\bullet \otimes_G k_g) =    \shTor_i^X(p^* \M, \N)\otimes_G k_g 
 \]
 for all $i \in \ZZ$ and $g \in V$.
 
Note that $\mu$ induces the  multiplication-by-$g$ isomorphism from $\{g\} \times X \to X$.  
The complex  $p^* \L_\bullet \otimes_G k_g$ is a locally free resolution of $\M$ on $\{g\} \times X$.
 Thus $\mu$ maps $\C_\bullet\otimes_G k_g$  to the complex $g_* \L_\bullet \otimes_X \N$,
 and these complexes are isomorphic considered as sheaves on $X$.  
The final complex computes $\shTor_i^X(g_* \M, \N)$.
It follows that
\[ 
\mu_*    (  \shTor_i^X(p^* \M, \N)\otimes_G k_g) \cong \shTor_i^X(g_* \M, \N)
\]
for all $i \in \ZZ$ and $g \in V$.

The sheaves $\shTor_i^X(p^* \M, \N)$ are flat over $V$ by assumption.
Since $V$ is an open subset of a connected variety, it is connected.
Thus for each $i$ it is the case that either the sheaf  $\shTor_i^X(p^* \M, \N)\otimes_G k_g$ is  0 for all $g \in V$ or that it is never $0$ for $g\in V$.
Replacing $g$ by $g^{-1}$, the result follows.
\end{proof}

\begin{proof}[Proofs of Theorem~\ref{thm1.3.2}]
By replacing $G$ by the Zariski closure of the subgroup generated by $\sigma$ we may assume that the forwards and backwards iterates of $\sigma$ are Zariski dense in $G$. By replacing $\sigma$ by an iterate, we may assume that $\sigma$ lies in the connected component of the identity of $G$ and that its iterates are dense in the connected component of the identity of $G$.  Since the result holds for $\sigma$ if it holds for an iterate of $\sigma$, we may then assume that $G$ is connected. Let $i\ge 1$. By Proposition \ref{prop:genericbehaviour}, there is a nonempty open subset $V$ of $G$ such that either
$\shTor_i^X(g^*\M,\N)$ is zero for every $g\in V$ or it is nonzero for every $g\in V$.  Now consider the map $f:G\to G$ given by $f(g)=\sigma\circ g$.  Then $f$ is an automorphism and $V$ is open, so the collection of integers $n$ for which $f^n(1)\in X\setminus V$ is a finite union of arithmetic progressions along with a finite set (cf. \cite{BGT}).  Moreover, since $S:=\{f^n(1)\colon n\in \mathbb{Z}\}$ is Zariski dense in $G$, we see that there cannot possibly be an infinite arithmetic progression upon which $f^n(1)\in X\setminus V$, since the Zariski closure of elements in this progression would be a proper closed subset and would be invariant under an iterate of $\sigma$, contradicting the fact that $S$ is dense in $G$. Thus there are only finitely many $n$ for which $f^n(1)\in X\setminus V$.  This means that $\shTor_i^X(\M^{\sigma^n},\N)$ is either zero for all but finitely many $n$ or it is nonzero for all but finitely many $n$.  Noting that we replaced $\sigma$ by an iterate, we then obtain the desired result.
 \end{proof}

\begin{appendices}
\section{Appendix: Module-Theoretic Analogue of Strassman's Theorem}
\label{sec:strassman}
Strassman's theorem \cite{strassman} 
is a cornerstone finiteness result in $p$-adic analysis, controlling the zeros of a convergent power series. Specifically, it states that if $K$ is a non-Archimedean field with valuation ring $\mathfrak{o}$, then every nonzero element $f(z)$ of the ring of convergent power series $K\langle z\rangle$ has only finitely many zeros in $\mathfrak{o}$. The goal of this appendix is to prove a module-theoretic analogue of Strassman's theorem.  Throughout this section, we take $K$ to be a non-Archimedean field containing $\QQ_p$, take $\mathfrak{o}$ to be its valuation ring, and $R$ to be a subring of $\mathfrak{o}$.

\begin{theorem}
Let $\mathcal{M}$ be a finitely generated $K\langle x_1,\ldots ,x_d,z\rangle$-module. Then the set of $c\in R$ for which $\mathcal{M}|_{z=c}=(0)$ is open in $R$. If $R$ is compact, then there exists $\varepsilon>0$ such that up to the addition and removal of finite sets, the set of $c\in R$ for which $\mathcal{M}|_{z=c}=(0)$ is a union of balls in $R$ of radius $\varepsilon$.

In particular, if $R=\mathbb{Z}$ then up to addition and removal of finite sets, the set of $c\in R$ for which $\mathcal{M}|_{z=c}=(0)$ is a finite union of arithmetic progressions with difference $p^r$ for some $r\ge 0$.
\label{theorem: strassman}
\end{theorem}

\begin{remark}
Using Theorem \ref{theorem: strassman}, we recover a weak version of Strassman's theorem, namely the case where $K/\QQ_p$ is a finite extension. Let $d=0$, $\M=K\langle z\rangle/(f(z))$, and $R=\mathfrak{o}$. Since $K/\QQ_p$ is finite, $R$ is compact and so Theorem \ref{theorem: strassman} tells us that there exists $\varepsilon>0$ such that up to addition and removal of finite sets, the set of $c\in R$ with $\M|_{z=c}=(0)$ is a union of $\varepsilon$-balls. Since there are only finitely many balls of radius $\varepsilon$, this set is a finite union of $\varepsilon$-balls, and hence closed.

Now notice that $\M|_{z=c}=(0)$ if and only if $f(c)\neq0$. As a result, up to addition and removal of finite sets, the set of $c\in R$ where $f$ vanishes is open. Since $f(z)$ has infinitely many zeros in $R$ by assumption, it must vanish on a ball, say of radius $\varepsilon'$. Recentering the ball, we can assume that $f(z)$ vanishes for all $|z|<\varepsilon'$. If $f(z)$ is non-zero, then it has the form $f(z)=a_nz^n+\sum_{i>n}a_iz^i$ with $a_n\neq0$. However, for $|z|$ sufficiently small, $a_nz^n$ has larger norm than $\sum_{i>n}a_iz^i$, and so $f$ does not vanish at $z$. We conclude that $f(z)=0$.
\end{remark}

Before proving Theorem \ref{theorem: strassman}, let us give some motivation as to why one would expect this statement to be true. We begin with a simple proof of Strassman's theorem using the commutative algebra of affinoids.  For each $\lambda\in \mathfrak{o}$, the ideal $(z-\lambda)$ is a maximal ideal of $K\langle z\rangle$.  If $f(\lambda)=0$ then $f(z)\in (z-\lambda)$. Suppose $f(\lambda)=0$ for all $\lambda\in \mathcal{S}$ where $\mathcal{S}$ is an infinite subset of $\mathfrak{o}$. Then $f(z)\in I:=\bigcap_{\lambda\in \mathcal{S}} (z-\lambda)$.  Note that $I$ is a radical ideal and it has finitely many minimal prime ideals  $P_1,\ldots ,P_r$ above it. Then since each $(z-\lambda)$ is above $I$, each $(z-\lambda)$ contains some $P_i$.  In particular, there is some $j$ and an infinite set of $\lambda$ such that $(z-\lambda)\supseteq P_j$.  But by Krull's principal ideal theorem we have that $(z-\lambda)$ is a height one prime ideal and since they are pairwise comaximal we then see that $P_j=(0)$ and so $I=(0)$.  

Taking this point of view, it is natural to seek an extension expressed in terms of ideal membership.  More precisely, one has an affinoid algebra $S:=K\langle x_1,\ldots ,x_d,z\rangle$ and an ideal $I(z)$.  One would like to conclude that if $f(z)\in S$ has the property that $f(z)\in I(z)+(z-c)S$ for infinitely many $c\in \mathfrak{o}$ then $f$ is in $I(z)$.  Notice that if one had such a statement, then taking $d=0$ and $I(z)$ to be the zero ideal gives Strassman's theorem. Unfortunately, this is false as stated. For example, the ideal $I(z)=(z)$ has the property that $1\in I(z)+(z-c)S$ for all nonzero $c\in \mathfrak{o}$. Another example, when $K$ is an extension of $\mathbb{Q}_p$, is given by the ideal $I(z):=(xz-1)$ in $K\langle x,z\rangle$, which has the property that $1\in I(z)+(z-c)R$ for all $c\in p\mathbb{Z}_p$. Notice that in the latter example, this set of $c$ is an open ball; in the former example it is an open ball minus a point.


We begin the proof of Theorem \ref{theorem: strassman} by showing that the desired set of $c$ is open in $R$.
%
\begin{lemma}
\label{l:Mc=0-is-open}
If $\mathcal{M}$ is a finitely generated $K\langle x_1,\ldots ,x_d,z\rangle$-module, then the set of $c\in R$ for which $\mathcal{M}|_{z=c}=(0)$ is open in $R$.
\end{lemma}
\begin{proof}
Now let $\mathcal{T}$ denote the collection of $c\in R$ for which $\M|_{z=c}=(0)$, or equivalently, $(z-c)\M=\M$. We show that $\mathcal{T}$ is open. To see this, let $g_1,\ldots ,g_\ell$ be a set of generators for $\M$ and let $c\in \mathcal{T}$.  Then $(z-c)h_i = g_i$ for some $h_i\in \M$ for $i=1,\ldots ,\ell$. We can write $h_i = \sum a_{i,j}g_i$.  Let $\varepsilon$ be a positive number that is less than the reciprocal of the Gauss norm of each nonzero $a_{i,j}$. 
Then for $\alpha\in B(c,\varepsilon)\cap R$, the ball of radius $\varepsilon$ centered at $c$, we have $(z-\alpha)h_i = (z-c)h_i+(c-\alpha)h_i=g_i + \sum (c-\alpha)a_{i,j}(z) g_j $.  
We can write this in the form
$$(z-\alpha)\left[ \begin{array}{c} h_1 \\ \vdots \\ h_\ell \end{array} \right] = 
(I + T)\left[ \begin{array}{c} g_1 \\ \vdots \\ g_\ell \end{array} \right],$$ where $T$ is a $\ell\times \ell$ matrix with entries in $K\langle x_1,\ldots ,x_d,z\rangle$ that are of Gauss norm is strictly less than $1$ and $I$ is the identity matrix. In particular, $I+T$ is invertible with inverse $I-T + T^2 - T^3 + \cdots \in \MM_d(K\langle x_1,\ldots ,x_d,z\rangle)$.  Thus multiplying our vector equation on the left by $(I+T)^{-1}$ we see that $(z-\alpha)h_i$ with $i=1,\ldots ,\ell$ generate $\M$.  Thus $B(c,\varepsilon)\subseteq\mathcal{T}$, showing that $\mathcal{T}$ is open.
\end{proof}

We next reduce the general case to that of cyclic modules.

\begin{lemma}
\label{l:Mc=0-is-closed-->cyclic}
Let $S=K\langle x_1,\ldots ,x_d\rangle$ and $\mathcal{M}$ be a finitely generated $S\langle z\rangle$-module. If $\N=S\<z\> / \ann(M)$. Then
\[
\{c\in R\mid \M|_{z=c}=(0)\} = \{c\in R\mid \N|_{z=c}=(0)\}.
\]
\end{lemma}
\begin{proof}
Let $\J=\ann(\M)$. First if $\N|_{z=c}=(0)$, then $1 = f+(z-c)g$ with $f\in\J$ and some $g\in S\langle z\rangle$, and so every $m\in\M$ can be expressed as $m=fm+(z-c)gm=(z-c)gm\in (z-c)\M$ showing that $\M|_{z=c}=(0)$.

Conversely, suppose $\M|_{z=c}=(0)$ and let $m_1,\dots,m_\ell$ be a set of generators of $\M$. Then there is a matrix $A\in\MM_\ell(S\<z\>)$ such that
\[
\left[ \begin{array}{c} m_1 \\ \vdots \\ m_\ell \end{array} \right] = (z-c)A\left[ \begin{array}{c} m_1 \\ \vdots \\ m_\ell \end{array} \right].
\]
Let $f=\det(I - (z-c)A)\in S\<z\>$ which is non-zero since it is of the form $1+(z-c)g$ with $g\in S\<z\>$, and $z-c$ is not a unit in $S\<z\>$. Letting $B$ be the adjugate matrix of $I-(z-c)A$, we know $B(I-(z-c)A)=fI$. Multiplying by $m_i$, we see $fm_i=0$ for all $i$, and so $f\in\J$. Thus, $1=f-(z-c)g\in\J+(z-c)S\<z\>$ showing that $\N|_{z=c}=(0)$.
\end{proof}

\begin{lemma}
\label{l:Mc=0-cyclic-->prime}
Let $S=K\langle x_1,\ldots ,x_d\rangle$. Given an ideal $I$ of $S\<z\>$, let $\M_I=S\<z\>/I$ and $\T_I$ be the set of $c\in R$ for which $M_I|_{z=c}=(0)$. If $J$ is the intersection of prime ideals $Q_1,\dots,Q_r$, then $\T_J=\bigcap_i\T_{Q_i}$. Moreover, if each $\T_{Q_i}$ is, up to addition and removal of finite sets, a union of $\varepsilon_i$-balls, then $\T_J$ is as well.
\end{lemma}
\begin{proof}
Notice that $c\in\T_I$ if and only if $(z-c)$ and $I$ are comaximal. For ease of notation, let $\T_i:=\T_{Q_i}$. It is clear that if $(z-c)+J=S\<z\>$, then $(z-c)+Q_i=S\<z\>$ for all $i$. Conversely, we have $1 = (z-c)a_i+q_i$ for $f_i\in S\<z\>$ and $q_i\in Q_i$. So $\prod_i (1-(z-c)a_i)\in J$ and it is of the form $1+(z-c)f$, so $(z-c)+J=S\<z\>$. This establishes the first statement.

Now suppose that there are finite sets $\mathcal{S}_{i1}$ and $\mathcal{S}_{i2}$ such that $(\T_i\cup\mathcal{S}_{i1})\cap\mathcal{S}_{i2}^c$ is a finite union of $\epsilon_i$-balls. Now, $\bigcap_i((\T_i\cup\mathcal{S}_{i1})\cap\mathcal{S}_{i2}^c)=(\bigcap_i\T_i\cup\bigcap_i\mathcal{S}_{i1})\cap\bigcap_i\mathcal{S}_{i2}^c$. Since $\bigcap_i\mathcal{S}_{i1}$ is a finite set and $\bigcap_i\mathcal{S}_{i2}^c$ is the complement of a finite set, we see that up to addition and removal of finite sets, $\T_J$ is the intersection of unions of $\varepsilon_i$-balls.

Without loss of generality, $\varepsilon_1\geq\dots\geq\varepsilon_r$. Let $(\T_i\cup\mathcal{S}_{i1})\cap\mathcal{S}_{i2}^c$ be the union of $\varepsilon_i$-balls $D_{i,j}$ where $j$ runs through an index set $\mathcal{S}_i$. Then up to addition and removal of finite sets, $\T_J$ is the union of sets of the form $D_{1,j_1}\cap\dots\cap D_{r,j_r}$ where $j_i\in\mathcal{S}_i$. Since $K$ is non-Archimedean, if two balls intersect, then one of them is contained in the other. Therefore, $D_{1,j_1}\cap\dots\cap D_{r,j_r}$ is either empty, or is an $\varepsilon_r$-ball, which finishes the proof.
\end{proof}

Finally, we handle the case where $\M$ is cyclic and $R$ is compact.

\begin{proposition}
\label{prop:Mc=0-is-closed}
Let $S=K\langle x_1,\ldots ,x_d\rangle$ and let $\mathcal{M}=S\<z\>/J$. If $R$ is compact, then there exists $\varepsilon>0$ such that, up to the addition and removal of finite sets, the set of $c\in R$ for which $\mathcal{M}|_{z=c}=(0)$ is a union of $\varepsilon$-balls.

In particular, it follows that if $R=\mathbb{Z}$ then up to addition and removal of finite sets, the set of $c\in R$ for which $\mathcal{M}|_{z=c}=(0)$ is a finite union of arithmetic progressions with difference $p^r$ for some $r\ge 0$.
\end{proposition}
\begin{proof}
Since there are only finitely many balls of a given radius in $\ZZ$, the statement concerning the case $R=\ZZ$ is immediate from the more general statement about compact $R$.

To prove the general claim about compact $R$, first notice that $\M|_{z=c}=(0)$ if and only if $(z-c)$ and $J$ are comaximal. Since $S\<z\>$ is Jacobson, this is equivalent to $(z-c)$ and $\rad(J)$ being comaximal, so we can assume $J$ is a radical ideal. Now, since $S\<z\>$ is Noetherian, $J$ is a finite intersection of prime ideals $Q_1,\dots,Q_r$. Then by Lemma \ref{l:Mc=0-cyclic-->prime}, we may assume that $J$ is prime.

By Noether normalization, $\M$ is a finitely generated module over some $T=K\<t_1,\dots,t_r\>$. Since $T$ is a domain, the stalk of $\M$ at the generic point of $\spec T$ is a vector space, hence free, and so there is an open neighborhood of $\spec T$ where $\M$ is free. So, there is some non-zero $q\in T$ such that the localization $\M_q$ is a free $T_q$-module, say of rank $n$. Then the action map of $\M$ on itself gives a map $\iota\colon\M\to\MM_n(T_q)$ which is an embedding since $\J$ is prime. Notice that $\M|_{z=c}=(0)$ if and only if $z-c\in\M$ is a unit if and only if $\det(\iota(z-c))\in T_q^*$. Indeed, it is clear that if $z-c$ is invertible then $\det(\iota(z-c))\in T_q^*$, and the converse follows from the Cayley-Hamilton theorem.

Let $t$ be an indeterminate and consider the minimal polynomial $f(t)$ of $\iota(z)$. Since the minimal polynomial divides the characteristic polynomial, there is a polynomial $g$ such that $f(t)g(t)=\det(\iota(z)-t)$. Away from the finitely many roots of $g(t)$, we see $f(c)\in T_q^*$ if and only if $\det(\iota(z-c))\in T_q^*$. Thus, we must prove that, up to addition and removal of a finite set, the set of $c$ for which $f(c)\in T_q^*$ is closed. In fact, we prove the stronger statement that there is an $\varepsilon>0$ such that the set of such $c$ is a finite union of $\varepsilon$-balls.

If there exists $c\in R$ with $f(c)=0$, then $f(t) = (t-c) r(t)$ in $\M$, and so $(z-c) r(z) = 0$ in $\M$. Since $\M$ is a domain, $z-c=0$ or $r(z)=0$ in $\M$.  Since $f$ is the minimal polynomial, the only case is that $f(t)=u(t-c)$, for some unit $u\in T_q^*$. But then $f(c') = u(c'-c)$, which is always a unit for $c\neq c'$, which proves the claim.

We now handle the case where $0\notin f(R)$. Since $R$ is compact, there is an $\varepsilon>0$ such that $||f(c)||>\varepsilon$. Let $f(t) = a_0+a_1t + \dots + a_d t^d$ with the $a_i\in R_q$; in fact, since $q\in R_q^*$ and we are only concerned with whether or not $f(c)\in R_q^*$, we can clear denominators and assume $a_i\in R$. Let $N$ be the maximum of the Gauss norm of the $a_i$. We show that if $f(c)$ is a unit and $|c-c'|<\varepsilon/(2Np^d)$, then $f(c')$ is a unit. To see why this is true, first consider the Taylor expansion $f(t)=f(c) + (t-c)r(t)$, where
\[
r(t) = f'(t) + (t-c) \frac{f''(t)}{2} + \dots + (t-c)^{d-1} \frac{f^{(d)}(t)}{d!}.
\]
We know that for $1\leq k\leq d$, the norm of $1/k$ is bounded by $p^d$. The coefficients of $f^{(k)}(t)$ are integer linearly combinations of the coefficients of $f(t)$, so have norm bounded by $N$. Since $R$ is a subring of the valuation ring, $|c'|$ and $|c-c'|$ are at most $1$. It follows that $|r(c')|\leq Np^d$, and so $|(c'-c)r(c')|<\varepsilon/2$. Since $|f(c)|>\varepsilon$, we see $(c'-c)r(c')/f(c)$ has norm at most $1/2$. Therefore,
\[
\frac{f(c')}{f(c)}=1+\frac{(c'-c)r(c')}{f(c)}
\]
is a unit.
\end{proof}

\end{appendices}

\end{document}